\documentclass{amsart}  
\usepackage{amscd,amssymb,amsmath,amsbsy,amsthm}
\usepackage{fge,tensor}
\usepackage{comment,color}
\usepackage[all]{xy}
\usepackage[colorlinks,plainpages,backref,urlcolor=blue]{hyperref}
\usepackage{breakurl}
\usepackage[width=5.65in,height=8.15in,centering]{geometry}
\usepackage{fancyvrb,bbm,blkarray}
\usepackage[mathscr]{euscript}
\usepackage{booktabs}
\usepackage{tikz}
\usetikzlibrary{calc,intersections,matrix,arrows}
\usepackage{mathrsfs}

\newtheorem{theorem}{Theorem}[section]
\newtheorem{lemma}[theorem]{Lemma}
\newtheorem{corollary}[theorem]{Corollary}
\newtheorem{proposition}[theorem]{Proposition}

\theoremstyle{definition}
\newtheorem{definition}[theorem]{Definition}

\newtheorem{exm}[theorem]{Example}
\newtheorem{rem}[theorem]{Remark}

\newtheorem*{ack}{Acknowledgments}

\newcommand{\mysetminus}{\mathbin{\fgebackslash}}

{\pushQED{\qed}\begin{exm}}%
{\popQED\end{exm}}  

\newenvironment{remark}%
{\pushQED{\qed}\begin{rem}}%
{\popQED\end{rem}}  

\newcommand{\abs}[1]{\left|#1\right|}
\newcommand{\set}[1]{\left\{#1\right\}}
\newcommand{\angl}[1]{\left<#1\right>}

\newcommand{\cT}{\mathscr{T}}  
\newcommand{\cE}{\mathscr{E}}
\newcommand{\Z}{\mathbb{Z}}
\newcommand{\M}{\mathsf{M}}    
\newcommand{\CP}{\mathbb{CP}}
\newcommand{\PP}{\mathbb{P}}
\newcommand{\DD}{\mathbb{D}}   
\newcommand{\Q}{\mathbb{Q}}
\newcommand{\R}{\mathbb{R}}
\newcommand{\C}{\mathbb{C}}
\newcommand{\T}{\mathbb{T}}
\renewcommand{\k}{\Bbbk}  
\newcommand{\one}{\mathbbm{1}}
\newcommand{\VN}{\mathcal{N}}  

\newcommand{\m}{\mathfrak{m}}

\DeclareMathOperator{\Tor}{Tor}

\DeclareMathOperator{\Ext}{Ext}
\DeclareMathOperator{\Hom}{Hom}

\DeclareMathOperator{\ab}{ab}

\DeclareMathOperator{\im}{im}
\DeclareMathOperator{\rank}{rank}

\DeclareMathOperator{\GL}{GL}  
\DeclareMathOperator{\PGL}{PGL}  
\DeclareMathOperator{\depth}{depth}
\DeclareMathOperator{\MCM}{MCM} 
\DeclareMathOperator{\OS}{OS}
\DeclareMathOperator{\SO}{SO}

\DeclareMathOperator{\gps}{Gps}  
\DeclareMathOperator{\stab}{stab}  
\DeclareMathOperator{\Funct}{Funct}  
\DeclareMathOperator{\codim}{codim}
\DeclareMathOperator{\pr}{pr}
\DeclareMathOperator{\rd}{red}   

\newcommand{\pairs}{{\mathcal P}} 

\newcommand{\A}{{\mathcal{A}}}
\newcommand{\B}{{\mathcal{B}}}
\newcommand{\F}{{\mathcal{F}}}
\newcommand{\G}{{\mathcal{G}}}
\renewcommand{\L}{{\mathcal{L}}} 
\newcommand{\CC}{{\mathcal{C}}} 
\newcommand{\V}{{\mathcal V}}

\newcommand{\N}{\mathscr{N}}  

\def\dot{\mathchar"013A}  
\newcommand{\hdot}{{\raise1pt\hbox to0.35em{\huge $\dot$}}} 
\def\set#1{\{#1\}}
\newcommand{\bwedge}{\mbox{$\bigwedge$}}

\newcommand{\cdga}{\textsc{cdga}}

\def\set#1{{\left\{#1\right\}}}

\begin{document}

\title[Local systems and arrangements of hypersurfaces]%
{Local systems on complements of arrangements of smooth, complex algebraic hypersurfaces}

\author[G. Denham]{Graham Denham$^1$} 
\address{Department of Mathematics, University of Western Ontario,
London, ON  N6A 5B7, Canada}  
\email{\href{mailto:gdenham@uwo.ca}{gdenham@uwo.ca}}
\urladdr{\href{http://www.math.uwo.ca/~gdenham}%
{http://www.math.uwo.ca/\~{}gdenham}}
\thanks{$^1$Research supported by NSERC of Canada}

\author[A.~I. Suciu]{Alexander~I.~Suciu$^2$}
\address{Department of Mathematics,
Northeastern University,
Boston, MA 02115, USA}
\email{\href{mailto:a.suciu@northeastern.edu}{a.suciu@northeastern.edu}}
\urladdr{\href{http://web.northeastern.edu/suciu/}%
{http://web.northeastern.edu/suciu/}}
\thanks{$^2$Partially supported by 
Simons Foundation collaborative grant 354156}

\begin{abstract}  
We consider smooth, complex quasi-projective varieties $U$ which admit
a compactification with a boundary which is an arrangement of smooth   
algebraic hypersurfaces.  If the hypersurfaces intersect locally like 
hyperplanes, and the relative interiors of the hypersurfaces are 
Stein manifolds, we prove that the cohomology of certain local systems 
on $U$ vanishes.  As an application, we show that complements of linear, 
toric, and elliptic arrangements are both duality and abelian duality spaces.
\end{abstract}



\maketitle

\section{Introduction}
\label{sect:intro}

\subsection{Abelian duality and local systems}
\label{subsec:intro1}
It has long been recognized that complements of complex hyperplane 
arrangements satisfy certain vanishing properties for homology with 
coefficients in local systems.  We revisited this subject in our joint 
work with Sergey Yuzvinsky, \cite{DSY1, DSY2},  in a more general context.  

Let $X$ be a connected, finite-type CW complex, with fundamental group $G$.  
Following Bieri and Eckmann \cite{BE}, we say that $X$ is a {\em 
duality space}\/ of dimension $n$ if $H^q(X,\Z{G})=0$ 
for $q\ne n$ and $H^n(X,\Z{G})$ is non-zero and torsion-free. 
We  also say that $X$ is an {\em abelian duality space}\/ of dimension 
$n$ if the analogous condition, with the coefficient $G$-module $\Z{G}$ 
replaced by $\Z{G}^{\ab}$ is satisfied. Setting $D:=H^n(X,\Z{G}^{\ab})$,  
it follows that $H^i(X,A)\cong H_{n-i}(G^{\ab},D\otimes_\Z{A})$
for any representation $A$ of $G$ which factors through  
$G^{\ab}$, and for all $i\ge 0$. 

Noteworthily, the abelian duality property imposes significant conditions 
on the cohomology of local systems on $X$.  Let $\k$ be an algebraically 
closed field.  The group $\widehat{G}=\Hom_{\gps}(G,\k^*)$ 
of $\k$-valued multiplicative characters of $G$ 
is an algebraic group, 
with identity the trivial representation $\one$.
The {\em characteristic varieties}\/ $\V^q(X,\k)$  are the 
subsets of $\widehat{G}$ consisting of those characters 
$\rho$ for which $H^q(X,\k_\rho)\neq0$.  
We highlight an interesting consequence of the abelian 
duality space property, which was established in \cite{DSY2}:   
If $X$ is an abelian duality space of dimension $n$,   
then the characteristic varieties of $X$ {\em propagate}, 
that is, 
\begin{equation}
\label{eq:cvprop}
\set{\one}=\V^0(X,\k)\subseteq \V^1(X,\k)\subseteq\cdots\subseteq\V^{n}(X,\k),
\end{equation}
or, equivalently, if $H^p(X,\k_\rho)\neq 0$ for some $\rho\in \widehat{G}$,
then $H^q(X,\k_\rho)\neq0$ for all $p\leq q\leq n$. 

The abelian duality property also imposes stringent conditions 
on the cohomology groups $H^i(X,\Z) = \Tor_{n-i}(D,\Z)$. For 
instance, all the Betti numbers $b_i(X)$ must be positive 
for $0\le i\le n$ and vanish for $i>n$, while $b_1(X)\ge n$.  
Furthermore, as noted in \cite[Theorem 1.8]{LMWb}, the results of \cite{DSY2} 
imply the following inequality for the `signed Euler characteristic'  
of an abelian duality space of dimension $n$:
\begin{equation}
\label{eq:signedeuler}
(-1)^n\chi(X)\geq 0.
\end{equation}

\subsection{Arrangements of smooth hypersurfaces}
\label{subsec:intro2}
Davis, Januszkiewicz, Leary, and Okun showed in \cite{DJLO11} 
that complements of (linear) hyperplane arrangements are duality 
spaces.  More generally, we proved in \cite{DSY2} that complements 
of both linear and elliptic arrangements are duality and abelian duality 
spaces.  

Our goal here is to further generalize these results to a much wider class 
of {\em arrangements of hypersurfaces}, by which we mean a 
collection of smooth, irreducible, codimension $1$ subvarieties 
which are embedded in a smooth, connected, complex projective 
algebraic variety, and which intersect locally like hyperplanes.  
We isolate a subclass of such arrangements
whose complements enjoy the aforementioned duality properties, 
and therefore have vanishing twisted cohomology in the appropriate range. 

\begin{theorem}
\label{thm:intro1}
Let $U$ be a connected, smooth, complex quasi-projective variety
of dimension $n$.  Suppose $U$ has a smooth compactification $Y$ for which
\begin{enumerate}
\item \label{c1}
Components of the boundary $D=Y\mysetminus U$ form 
a non-empty arrangement of hypersurfaces $\A$;
\item  \label{c2}
For each submanifold $X$ in the intersection poset $L(\A)$,  
the complement of the restriction of $\A$ to $X$ is either 
empty or a Stein manifold. 
\end{enumerate}
Then $U$ is both a duality space and an abelian duality space 
of dimension $n$.
\end{theorem}

An important consequence of this theorem is that the characteristic varieties 
of such ``recursively Stein'' hypersurface complements propagate. 
As another application of Theorem~\ref{thm:intro1}, we prove in Corollary 
\ref{cor:generic vanishing} 
the following `generic vanishing of cohomology' result.  
We use here De Concini and Procesi's \cite{DP95} construction of wonderful 
models for subspace arrangements, based on the notion of `building sets'; 
see \cite{Fe05} for an exposition. 

\begin{theorem}
\label{thm:genvanish}
Let $U$ be a smooth, quasi-projective variety of dimension $n$ satisfying 
conditions \eqref{c1} and \eqref{c2} from above.  Furthermore, let  
$G=\pi_1(U)$, and let $A$ be a finite-dimensional representation
of $G$ over a field $\k$.  Suppose that $A^{\gamma_g}=0$ for all
$g$ in a building set $\G_X$, where $X\in L(\A)$;   
then $H^i(U,A)=0$ for all $i\neq n$.
\end{theorem}

Consequently, the cohomology groups of $U$ with coefficients in a 
`generic' local system vanish in the range below $n$.

Finally, let $\ell_2 G$ denote the left $\C[G]$-module of 
complex-valued, square-summable functions on $G$, 
and let $\tensor[^\rd]{H}{^i}(U,\ell_2 G)$ be the reduced 
$L^2$-cohomology groups of $U$ with coefficients in this module, 
cf.~ \cite{Eck, Lueck}.  We then prove in Theorem \ref{thm:l2} 
the following result. 

\begin{theorem}
\label{thm:l2vanish}  
Let $U$ and $G=\pi_1(U)$ be as above.  Then 
$\tensor[^\rd]{H}{^i}(U,\ell_2 G)=0$ for all $i\neq n$.
\end{theorem}

In particular, the $\ell_2$-Betti numbers of $U$ are all zero except
in dimension $n$.  A basic fact about $\ell_2$-cohomology is that
$\ell_2$-Betti numbers compute the usual Euler characteristic (see, e.g.,
\cite[Thm.\ 3.6.1]{Eck}.)  Therefore, we see once again that 
$(-1)^n\chi(U)\geq 0$.

\subsection{Linear, elliptic, and toric arrangements}
\label{subsec:intro3}
The theory of hyperplane arrangements originates in the study of configuration
spaces and braid groups.  Here we consider a broader class of hypersurface
arrangements of current interest.

\begin{theorem}
\label{thm:intro2}
Suppose that $\A$ is one of the following:
\begin{enumerate}
\item \label{a1} 
An affine-linear arrangement in $\C^n$, or a hyperplane arrangement in $\CP^n$; 
\item \label{a2} 
A non-empty elliptic arrangement in $E^n$;
\item \label{a3} 
A toric arrangement in $(\C^*)^n$.
\end{enumerate}
Then the complement $M(\A)$ is both a duality space and an abelian duality
space of dimension $n-r$, $n+r$, and $n$, respectively, where $r$ is the
corank of the arrangement. 
\end{theorem}

As mentioned previously, the first two statements already appeared 
in our paper \cite{DSY2}; at the time, however, we were unable to 
address the third one.  Since then, 
De Concini and Gaiffi \cite{DG16} have constructed a compactification 
for toric arrangements which is compatible with our approach.
The claim that the complement of a toric arrangement is a duality space
was first reported in \cite[Theorem 5.2]{DS13}.  
However, a serious gap appeared in the proof: see \cite{Da16}.  
Part of our motivation here, then, is to provide an independent alternative,
as well as a uniform proof of the three claims above.  The argument
is given in \S\ref{subsec:toric} and depends on Theorem~\ref{thm:intro1}.
It is worth noting that this dependency is slightly subtle.  
For example, the complement $U$ of a 
linear hyperplane arrangement is indeed a Stein manifold; however, if 
the corank is strictly positive, hypothesis (2) is not satisfied by
the intersection of all the hyperplanes.

As a consequence of Theorem~\ref{thm:intro2}, the characteristic varieties 
propagate for all linear, elliptic and toric arrangements.  The formality of 
linear and toric arrangement complements implies that their resonance 
varieties propagate, as well.  In the linear case, a more refined propagation
of resonance property was established by Budur in \cite{Bu11}.

If $\A$ is an affine complex arrangement, work of Kohno \cite{Ko86}, 
Esnault, Schechtman, Varchenko \cite{ESV}, and Schechtman, 
Terao, Varchenko \cite{STV} gives sufficient conditions for a 
local system $\L$ on $M(\A)$ to insure the vanishing of the 
cohomology groups $H^i(M(\A),\L)$ for all $i<\rank(\A)$.  
Similar conditions for the vanishing of cohomology of with 
coefficients in rank $1$ local systems were given by Levin and 
Varchenko \cite{LV} for elliptic arrangements, and by Esterov 
and Takeuchi \cite{ET} for certain toric hypersurface arrangements.  
In turn, we provide in Corollary \ref{cor:generic vanishing} a unified 
set of generic vanishing conditions for cohomology of  
local systems on complements of arrangements of 
smooth, complex algebraic hypersurfaces. 

The $\ell_2$-cohomology of the complement of a linear arrangement 
also vanishes outside of the middle (real) dimension: this is a result of 
Davis, Januszkiewicz and Leary~\cite{DJL07}.  The same claim for toric
arrangements appears in \cite{DS13}; however, the argument given there 
has the same gap mentioned above.  As part of our approach here,
we obtain the aforementioned vanishing result (Theorem~\ref{thm:l2}) 
for the $\ell_2$-cohomology of complements of hypersurface arrangements.  

\subsection{Orbit configuration spaces}
\label{subsec:intro4}

As a second application of our general results, we obtain an almost complete 
characterization of the duality and abelian duality properties of ordered 
orbit configuration spaces on Riemann surfaces.  In \S\ref{subsec:conf}, 
we define and discuss orbit configuration spaces, following \cite{Xi97}; 
for the purpose of this introduction, though, we remind the reader that 
the classical configuration spaces are recovered by taking $\Gamma$ 
to be the trivial group.

\begin{theorem}
\label{thm:intro3}  
Suppose $\Gamma$ is a finite group that acts freely 
on a Riemann surface $\Sigma_{g,k}$ of genus $g$ with $k$ punctures. 
Let $F_{\Gamma}(\Sigma_{g,k},n)$ be the orbit 
configuration space of $n$ ordered, disjoint $\Gamma$-orbits.  
\begin{enumerate}
\item If $k>0$, then
$F_{\Gamma}(\Sigma_{g,k},n)$ is both a duality space and an abelian duality
space of dimension $n$.
\item If $k=0$, then $F_{\Gamma}(\Sigma_g,n)$ is a duality space of dimension
$n+1$, provided $g\geq1$, and is an abelian duality space of dimension $n+1$ if
$g=1$.
\item If $k=0$, then $F(\Sigma_g,n)$ is neither a duality space nor an abelian
duality space if $g=0$, and it is not an abelian duality space if $g\geq2$.
\end{enumerate}
\end{theorem}

Hence the characteristic varieties propagate for the 
orbit configuration spaces $F_{\Gamma}(\Sigma_{g,k},n)$, 
where either $k\geq1$, or $k=0$ and $g=1$, for any finite group 
$\Gamma$ acting freely on  $\Sigma_{g,k}$.

\section{Hypersurface arrangements}
\label{sect:arrs}

\subsection{Open covers for hypersurface arrangements}
\label{subsec:sub mfd}
Let $Y$ be a smooth, connected complex manifold, and let 
$\A=\set{W_1,\ldots, W_m}$ be a finite collection of smooth, connected, 
codimension-$1$ submanifolds of $Y$.  Let $D=\bigcup_{i=1}^m W_i$ 
be the corresponding divisor, and let $M(\A):=Y\mysetminus D$ be 
the {\em complement}\/ of the arrangement $\A$.

We will assume that the intersection of any subset of $\A$ 
is also a smooth manifold, and has only finitely many connected 
components. Moreover, we require that, for each point $y\in D$, 
there is a chart containing $y$ for which each element of the subcollection 
$\A_y:=\set{W_i\mid y\in W_i}$ is defined locally by a linear equation.
In other words, the hypersurfaces comprising $\A$ have intersections 
which, locally, are diffeomorphic to hyperplane arrangements.  
This definition is taken from Dupont~\cite{Du15}, though similar 
notions appear in the literature; in particular, we refer to the 
paper of Li~\cite{Li09}.  

Let $L(\A)$ denote the collection of all connected components of
intersections of zero or more of the hypersurfaces comprising $\A$.  
Then $L(\A)$ forms a finite poset under reverse inclusion, ranked by 
codimension.  We will write $X\leq Y$ if $X\supseteq Y$, and write 
$r(X)=\codim X$ for the 
rank function.  For every submanifold $X\in L(\A)$, let 
$\A_X=\set{W\in\A\mid X\subseteq W}$
be the {\em closed subarrangement}\/ associated to $X$, and let 
$\A^{X} = 
\set{W\cap X \mid W\in \A\mysetminus \A_X}
$
be the {\em restriction}\/ of $\A$ to $X$. 
We write 
\begin{equation}
\label{eq:dx}
D_X=\bigcup_{Z\in L(\A) : {Z<X}} Z.
\end{equation}
Then the complement of the restriction of $\A$ to $X$ is $M(\A^X):=
X\mysetminus D_X$, for each $X\in L(\A)$.
Finally, let $T\A_X$ be the hyperplane
arrangement in the tangent space to $Y$ at 
a point in the relative interior of $X$, guaranteed by our hypothesis on the 
intersection of hypersurfaces.

One of the main tools we will need in this note is a 
spectral sequence developed in \cite{DSY1}, which we 
summarize in the next theorem.

\begin{theorem}[\cite{DSY1}]
\label{thm:ss-dsy}
Let $\A$ be an arrangement of hypersurfaces in a compact, smooth manifold $Y$.
Let $M$ be the complement of the arrangement, 
and let $\F$ be a locally constant sheaf on $M$.  
There is then a spectral sequence with
\begin{equation}
\label{eq:ss arr}
E_2^{pq}=\prod_{X\in L(\A)} H^{p+r(X)}_c(M(\A^X);
H^{q-r(X)}(M(T\A_X),\F_X)),
\end{equation}
converging to $H^{p+q}(M,\F)$, where
$\F_X$ is the corresponding restriction of $\F$ to $M(T\A_X)$.
\end{theorem}  
\begin{remark}
This is a special case of \cite[Cor.~3.3]{DSY1}.  The indexing differs
slightly since \cite{DSY1} indexes by dimension, rather than codimension.
\end{remark}

\subsection{Wonderful compactifications}
\label{subsec:wonder}

From here on, we will consider arrangements $\A$ of smooth, algebraic 
hypersurfaces in a smooth, connected complex projective variety $Y$.
For each $x\in Y$, there is a linear hyperplane arrangement $T\A_X$ in
the complex vector space $V=T_xY$ tangent to $\A_X$, 
where 
\begin{equation}
\label{eq:xx}
X=\bigcap_{x\in Z\in L(\A)}Z.
\end{equation}

We apply De Concini and Procesi's construction of the wonderful 
model of a subspace arrangement \cite{DP95} to the linear arrangement 
$T\A_X$ inside the affine space $V$.  The construction blows up the 
arrangement to one with simple
normal crossings; let $p\colon \widetilde{V}\to V$ denote the blowup.
The (total) divisor components are indexed by a `building set' $\G_X$.
A subset $S\subseteq\G_X$ indexing divisor components that have non-empty
intersection is called a `nested set,' and the collection of all nested sets 
forms a simplicial complex, called the nested set complex.  
The nested set complex $\N(T\A_X)$ depends on the choice of
building set; however, we will assume a fixed choice is made for each $X$
and omit the building set from the notation, since the choices are not
important to what follows.

\subsection{An injectivity result}%
\label{subsec:inj}
We recall from \cite{DSY1} that the fundamental group of the complement 
of a linear arrangement contains certain distinguished  free 
abelian subgroups of finite rank.  We review the definition of 
these subgroups briefly, referring to \cite{DSY1,Fe05} for 
unexplained terminology.

We consider the wonderful model for $T\A_X$ in the tangent space 
$V=V_x$ for a point $x$ in the relative interior of $X$.  
For a nested set $S\in \N(T\A_X)$ of size $r$, 
let $D_S$ denote the corresponding intersection of $r$
divisor components in $\widetilde{V}$.  For a point $z$ in the relative
interior of $D_S$, let $\DD_z$ be a sufficiently small closed polydisc in 
$\widetilde{V}$ centered at $z$.  Then 
$U_S:=\DD_z\cap M(T\A_X)\simeq (S^1)^r$.
It is shown in \cite[Thm.~4.6]{DSY1} that the inclusion 
$f_{X,S}\colon U_S\hookrightarrow M(T\A_X)$ induces an injective map of 
fundamental groups, $(f_{X,S})_{\sharp}\colon \pi_1(U_S,z_{X,S})
\hookrightarrow \pi_1(M(T\A_X),z_{X,S})$.   
Let $C_S$ denote the image of this homomorphism.  This 
is  a free abelian subgroup of
$G_X:=\pi_1(M(T\A_X))$ of rank $r$, well-defined up to conjugacy.

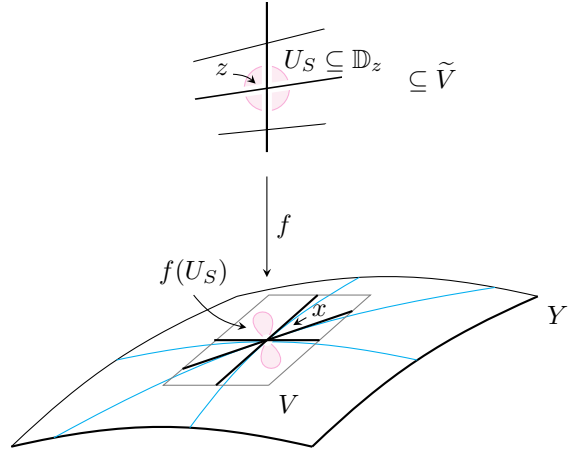
\begin{figure}
\begin{tikzpicture}[vertex/.style={circle,draw,inner sep=1pt,fill=black}]
\coordinate (SW) at (-2,-2);
\coordinate (SE) at (2,-2);
\coordinate (NW) at (1,0);
\coordinate (NE) at (5,0);
\draw[thick] (SW) to [bend left=13] coordinate[pos=0.13] (SL) 
coordinate[pos=0.6] (S) (SE);
\draw (NW) to [bend left=13] coordinate[pos=0.87] (NR) 
coordinate[pos=0.4] (N) (NE);
\draw (SW) to [bend left=10] coordinate[pos=0.5] (W) (NW);
\draw[thick] (SE) to [bend left=10] coordinate[pos=0.5] (E) (NE);
\draw[cyan] (S) to [bend left=12] (N);
\draw[cyan] (W) to [bend left=12] (E);
\draw[cyan] (SL) to [bend left=10] coordinate[pos=0.5] (Z) (NR);
\coordinate (VNE) at ($(Z)+(0.7,0)+0.3*(S)-0.3*(N)$);
\coordinate (VNW) at ($(Z)+(0.7,0)-0.3*(S)+0.3*(N)$);
\coordinate (VSE) at ($(Z)-(0.7,0)+0.3*(S)-0.3*(N)$);
\coordinate (VSW) at ($(Z)-(0.65,0)-0.3*(S)+0.3*(N)$);
\draw[gray] (VNE) -- (VNW) -- (VSW) -- (VSE) -- (VNE);
\filldraw[magenta,fill=magenta!30!white,opacity=0.3] (Z) .. controls
($(Z)-0.3*(S)+0.17*(N)$) and ($(Z)-0.27*(S)-0.17*(N)$) .. (Z);
\filldraw[magenta,fill=magenta!30!white,opacity=0.3] (Z) .. controls
($(Z)+0.3*(S)-0.17*(N)$) and ($(Z)+0.27*(S)+0.17*(N)$) .. (Z);
\draw[white,fill=white] (Z) circle (0.07);
\draw[thick] ($(Z)-(0.7,0)$) to ($(Z)+(0.7,0)$);
\draw[thick] ($(Z)-0.3*(S)+0.3*(N)$) to ($(Z)+0.3*(S)-0.3*(N)$);
\coordinate (TNR) at ($(Z)+(NR)-(SL)$);
\coordinate (TSL) at ($(Z)-(NR)+(SL)$);
\draw[thick] (intersection of TSL--TNR and VNE--VNW) to
             (intersection of TSL--TNR and VSE--VSW);
\begin{scope}
\def\dy{3.5}
\clip ($(Z)+(0,\dy)$) circle (1);
\coordinate (ES) at ($(Z)+(0,-1+\dy)$);
\coordinate (EN) at ($(Z)+(0,1+\dy)$);
\coordinate (TS) at ($(Z)+(-1,-0.8+\dy)$);
\coordinate (TN) at ($(Z)+(1,-0.6+\dy)$);
\coordinate (TW) at ($(Z)+(-1,-0.3+\dy)$);
\coordinate (TE) at ($(Z)+(1,0+\dy)$);
\coordinate (TSL) at ($(Z)+(-1,0.2+\dy)$);
\coordinate (TNR) at ($(Z)+(1,0.7+\dy)$);
%
%
\coordinate (X) at (intersection of ES--EN and TE--TW);
\draw[magenta,fill=magenta!30!white,opacity=0.3] (X) circle (0.3);
\draw[ultra thick,white,double=black] (ES) to (EN);
\draw[ultra thick,white,double=black] (TE) to (TW);
%
\draw[thick] (ES) to (EN);
\draw (TS) to (TN);
\draw (TSL) to (TNR);
\end{scope}
\node[above right] at ($(X)+(0.1,0.1)$) {$U_S\subseteq \DD_z$};
\coordinate (ZZ) at ($(X)+(-0.6,0.25)$);
\node at (ZZ) {$z$};
\path[-stealth] ($(X)!0.35!(Z)$) edge node [right] {$f$} ($(X)!0.75!(Z)$);
\node[below right] at (NE) {$Y$};
\node[right] at ($(TE)+(0.75,0)$) {$\subseteq \widetilde{V}$};
\node[below right] at (VNE) {$V$};
\coordinate (XX) at ($(Z)+(0.7,0.4)$);
\node at (XX) {$x$};
\draw[-stealth] ($(ZZ)!0.25!(X)$) to [bend left=20] ($(ZZ)!0.75!(X)$);
\draw[-stealth] ($(XX)!0.25!(Z)$) to  ($(XX)!0.5!(Z)$);
\coordinate (fU) at ($(VSW)-(1,0)$);
\node[above] at (fU) {$f(U_S)$};
\draw[-stealth] (fU) to [bend right=20] ($(fU)!0.72!(Z)$);
\end{tikzpicture}
\label{fig:ball in blowup}
\caption{Small tori in $M(\A)$}
\end{figure}

Our goal now is to show that $C_S$ also injects in the fundamental group of
the hypersurface complement.  
Let $\DD$ be a closed polydisc in $V=V_x$ centered at the origin, and set 
$V_X:=\DD\mysetminus \bigcup_{W\in\A_X}T_{x}W\cap\DD$.  
Let $h$ denote the diffeomorphism that trivializes $TY|_{\DD}$,
and let $\DD_x$ be the image of $\DD$ under $h$.  Let $U_X=h(V_X)$;
then $U_X=\DD_x\mysetminus \bigcup_{W\in\A_X}W\cap\DD_x$.

Consider the diagram
\begin{equation}
\label{eq:cd}
\begin{gathered}
\begin{tikzpicture}
\matrix (m) [matrix of math nodes,row sep=3em,column sep=3em,minimum width=2em]
{
U_S &  M(T\A_X) & V_X & 
U_X & M(\A) \phantom{\, .}\\
& & & \DD & Y \, . \\ };   
\path[right hook-stealth] 
(m-1-1) edge node [above] {$f_{X,S}$} (m-1-2);
\path[-stealth]
(m-1-2) edge node [above] {$g$} (m-1-3);
\path[-stealth]
(m-1-3) edge node [above] {$h$} (m-1-4);
\path[right hook-stealth]
(m-1-4) edge  (m-1-5)
edge  (m-2-4)
(m-1-5) edge node [right] {$j$} (m-2-5)
(m-2-4) edge node [above] {$i$} (m-2-5);
\end{tikzpicture}
\end{gathered}
\end{equation}

Here, $g$ is the usual deformation retraction of a central, 
linear arrangement complement, and the 
composite of the maps on the top row, $U_S\to M(\A)$, 
induces a homomorphism $\alpha_{X,S}\colon\pi_1(U_S)\to \pi_1(M(\A))$. 

Our argument for the injectivity of this map uses some rational 
homotopy theory.  We start with a technical lemma.  A connected 
$\cdga$ $(A,d)$ is said to be \emph{$1$-minimal}\/ 
if $A=\bwedge V$, the free $\cdga$ on a vector space $V$ 
concentrated in degree $1$, and $V$  is the union of an increasing 
filtration by subspaces $\{V_i\}_{i\ge 0}$ such that  $d=0$ on $V_0$ 
and $d(V_{i+1})\subset \bwedge V_i$.  (The differential $d$ is then 
decomposable, and thus $A$ is a minimal Sullivan algebra in the 
sense of \cite{FHT}.)

\begin{lemma}
\label{lem:surj}
Let $A$ be a $1$-minimal $\cdga$, and let $\varphi, \psi\colon A\to B$ 
be two homotopic $\cdga$ morphisms.  
Suppose $d_B=0$ and  $\varphi^1\colon A^1\to B^1$ is surjective.  
Then $\psi^1\colon A^1\to B^1$  is also surjective.  
\end{lemma}

\begin{proof}
Suppose $\psi^1$ is not surjective. Let us view $C:=\k\oplus B^1/\im(\psi^1)$
as a (connected) $\cdga$ with zero differential, and let $\pr \colon B \to C$ be the 
projection map, which sends the graded ideal generated by $\im (\psi^1)$ to zero.
Clearly, the $\cdga$ maps $\Phi=\pr\circ \varphi$ 
and $\Psi= \pr\circ \psi$ are homotopic, 
and $\Psi=0$; thus, $\Phi$ is null-homotopic.  By \cite[Example~1, p.~151]{FHT}, 
the map $\Phi$ is constant: that is, zero in positive degree. This 
implies that $\varphi^1$ is not surjective, a contradiction.
\end{proof}

The complement of an arrangement of hypersurfaces $\A$ in a smooth, 
projective variety $Y$ has a rational $\cdga$ model, given by the $E_2$ 
page of the Leray spectral sequence of the inclusion $j\colon M(\A)\to Y$. 
In the normal crossings case, this is the classical Morgan model; for 
configuration spaces it was used by Totaro \cite{To96}; in our more 
general context, we refer to Dupont~\cite{Du15} and Bibby~\cite{Bi16}.  
We will denote this model for $M(\A)$ by $B(\A)$.

\begin{lemma}
\label{lem:inj}
Let $\A$ be an arrangement of smooth, complex subvarieties 
in a smooth, complex projective variety $Y$. Then, for every 
$X\in L(\A)$ and every nested set $S\in \N(T\A_X)$, the 
homomorphism $\alpha_{X,S}\colon\pi_1(U_S)\to \pi_1(M(\A))$ 
is injective.
\end{lemma}

\begin{proof}
Restriction along the map $i\colon \DD\to Y$ 
gives a map of sheaves $\Q\to i_*\Q$ on $M(\A)$,
which in turn gives a map between the $E_2$-pages of the 
respective Leray spectral sequences, 
\begin{equation}
\label{eq:istar}
\xymatrixcolsep{20pt}
\xymatrix{i^*\colon H^p(Y,R^qj_*\Q)\ar[r]& H^p(\DD,R^qj_*\Q)}.
\end{equation}

We note that, for each $q\geq0$, the
sheaf $R^qj_*\Q$ is the direct image of a constant sheaf
on a contractible space, cf.~\cite[Lemma~3.1]{Bi16}; it follows that 
$H^p(\DD,R^qj_*\Q)=H^q(U_X,\Q)$ for $p=0$, and is zero 
otherwise.  Thus, the $E_2$-page on the right side of \eqref{eq:istar} 
is a $\cdga$ with zero differential, isomorphic to the Orlik--Solomon algebra 
$\OS_X:=H^\hdot(M(T\A_X),\Q)$, and a rational
model for $U_X$.

Now recall from \cite{Du15} that the $\cdga$ $B(\A)$ contains distinguished 
elements in bidegree $(0,1)$, which we 
denote by $\set{e_W\mid W\in\A}$, that restrict locally to logarithmic
$1$-forms around the components of $\A$ in $Y$.
Since their images $i^*(e_W)$ for $W\in\A_X$ are the 
standard generators of $\OS_X$, the map $i^*$ is surjective. 
Recall also that we identified the group $\pi_1(U_S)$ with $\Z^r$, 
for some $r>0$.  By \cite[Proposition 2]{FY04}, or the proof of 
\cite[Corollary 4.7]{DSY1}, the map $f^*_{X,S}\colon 
\OS_X\to H^\hdot((\C^*)^{r},\Q)$ is surjective in degree $1$.  
Hence, the composite $p=f_{X,S}^*\circ i^*\colon B(\A)\to \bwedge(\Q^{r})$ 
is a surjective map of $\cdga$s, where the differential in $\bwedge(\Q^{r})$ is zero.

As is well-known, every $\cdga$ $A$ has a $1$-minimal model, $\M(A)$, which 
is unique up to homotopy, see e.g.~\cite{GM}.  Such a model 
comes equipped with a morphism $\pi\colon \M(A)\to A$ 
inducing an isomorphism on $H^1$ and a monomorphism on $H^2$.  
We thus obtain a commuting square, 
\begin{equation}
\label{eq:minmod}
\begin{gathered}
\xymatrixrowsep{32pt}
\xymatrix{
\M(B(\A))\ar[d]^{\pi}\ar[r] \ar@{.>}[dr]^{\psi} \ar@/^1pc/@{.>}[d]^{\pi'} 
&\M(\bwedge(\Q^{r}))\ar^{=}[d]\\
B(\A)\ar[r]^{p} & \bwedge(\Q^{r}) \, ,
}
\end{gathered}
\end{equation}
noting that the rational exterior algebra is its own minimal model. 
Let $\varphi=p\circ \pi$.  
Since both $H^1(p)$ and $H^1(\pi)$ are surjective, the map 
$H^1(\varphi)$ is also surjective.  Since the differential 
of $\bwedge(\Q^{r})$ is zero, the map $\varphi$ itself is surjective 
in degree $1$.

Let $\m(G)$ be the pronilpotent Lie algebra associated to $G=\pi_1(M(\A),x)$.  
As it is again well-known, $\m(G)$ coincides with the Lie algebra 
dual to the $1$-minimal model $\M(B(\A))$; we refer to \cite{GM} 
and also \cite[\S7]{SW} for discussion and further references.  
Let $\psi \colon \M(B(\A)) \to  \bwedge(\Q^{r})$ be the $\cdga$ 
morphism dual to the Lie algebra map $\m(\alpha_{X,S})\colon \m(\Z^r)\to \m(G)$. 
By minimality of $\M(B(\A))$, the map $\psi$ lifts to a map 
$\pi'\colon \M(B(\A))\to B(\A)$, see \cite[Lemma 12.4]{FHT}. 
By construction, $\pi'$ is a classifying map for the $1$-minimal model of $B(\A)$; 
thus, $\pi\simeq \pi'$, and so $\varphi\simeq \psi$.    
By Lemma \ref{lem:surj}, the map $\psi$ is surjective in degree $1$;  
hence, $\m(\alpha_{X,S})$ is injective. 

Let $(\alpha_{X,S})_{\Q}\colon (\Z^r)_\Q\to G_\Q$
be the induced homomorphism between rational, prounipotent 
completions.  By the above, this homomorphism is also injective.  
Since $\Z^r$ is a finitely generated, torsion-free abelian
group, the natural map $\iota=\Z^r\to (\Z^r)_\Q$ is injective, as well.   
Hence, the map $\alpha_{X,S}=(\alpha_{X,S})_{\Q} \circ \iota$ 
itself is injective, and we are done.
\end{proof}

\subsection{The main result}
\label{subsec:main}
The cohomological vanishing results in \cite{DSY1} 
made use of the following condition on $G$-modules.

\begin{definition}
\label{def:mcm}
Let $\k=\Z$ or a field, let $R=\k[\Z^n]$ for some $n\geq1$, and let $I_\one$ be 
the augmentation ideal of $R$.  We say that an $R$-module $A$ is a 
{\em maximal Cohen--Macaulay ($\MCM$) module}\/ provided that 
$\depth_R(I_\one,A)\geq n$.  We note that this is slightly weaker than the 
usual notion, since we do not assume $A$ is finitely generated, and we 
allow the possibility that $\Ext_R(R/I_\one,A)$ is identically zero, in which
case we take $\depth_R(I_\one,A)=\infty$.
\end{definition}

Now let $\A$ be an arrangement of smooth, complex subvarieties 
in a smooth, complex variety $Y$, and let $G=\pi_1(M(\A))$.  From now on, 
we will assume $Y$ is compact.  In view of Lemma~\ref{lem:inj}, we will let 
$C_{S,X}$ denote the conjugacy class of the subgroup
$\alpha_{X,x}(C_S)<G$, a free abelian group of rank $\abs{S}$.
We then extend the previous definition to this context, as follows.  

\begin{definition}
\label{def:mcm-arr}
Let $A$ be a (left) $\k[G]$-module.  We say that $A$ is a {\em $\MCM$ module}\/ 
provided that the restriction of $A$ to each subalgebra $\k[C_{S,X}]$ is $\MCM$, for
all flats $X\in L(\A)$ and all nested sets $S\in \N(T\A_X)$.
\end{definition}

Before proceeding, we need to recall some well-known facts about 
Stein manifolds, see e.g. \cite{FG02, For}.   A complex manifold $M$ 
is said to be a {\em Stein manifold}\/ if it can be realized as a closed, 
complex submanifold of some complex affine space.  Alternatively, 
holomorphic functions on $M$ separate points, and $M$ is holomorphically 
convex.
The Stein property 
is preserved under taking closed submanifolds and finite direct products. 
Furthermore, every Stein manifold of (complex) dimension $n$ 
has the homotopy type of a CW-complex of dimension~$n$.

We are now in a position to state and prove our main vanishing-of-cohomology 
result. 

\begin{theorem}
\label{thm:mcm}
Let $\A$ be an arrangement of hypersurfaces in a compact, complex, 
smooth variety $Y$ of dimension $n$. 
Suppose that $M(\A^X)$ is Stein for each $X\in L(\A)$.  Then, for any
$\MCM$ module $A$ on $M(\A)$, we have $H^p(M(\A),A)=0$ for all $p\neq n$.
\end{theorem}

\begin{proof}
We use Theorem~\ref{thm:ss-dsy} and imitate the proof of \cite[Thm.\ 6.3]{DSY1}.
For each $X\in L(\A)$, we recall that the restriction of the Hopf
fibration $\C^*\to M(T\A_X)\to U(T\A_X)$ identifies a central, cyclic subgroup
in $G_X$; let $\gamma_X\in Z(G_X)$ be a generator.
By the $\MCM$ hypothesis and the discussion from \cite[\S4.2]{DSY1}, 
the coinvariant module $A_{\gamma_X}$ is a $\MCM$ module over  $\pi_1(U(T\A_X))
=G_X/\left<\gamma_X\right>$, and
\begin{equation}
\label{eq:him}
H^i(M(T\A_X),A) \cong H^{i-1}(U(T\A_X),A_{\gamma_X}) 
\end{equation}
for all $i> 0$.  Furthermore, this latter group vanishes for $i\neq r(X)$, 
by \cite[Thm.\ 5.6]{DSY1}.
By hypothesis, though, $M(\A^X)$ is a Stein
manifold of (complex) dimension $n-r(X)$. Thus, the factor of
\eqref{eq:ss arr} indexed by $X$ is zero for $p+r(X)<n-r(X)$.  

Combining these facts, we see that $E_2^{pq}=0$, unless $p+q\geq n$; 
therefore, $H^{p+q}(M,A)=0$, unless $p+q\geq n$.  On the other hand, 
$M=M(\A)=M(\A_{\widehat{0}})$ 
is itself Stein, so $H^{p+q}(M,A)=0$ unless $p+q\leq n$.  The conclusion 
follows.
\end{proof}

\begin{remark}
\label{rem:non-stein}
The Stein hypothesis in Theorem \ref{thm:mcm} 
is indispensable.  Indeed, let $X=\C^n$, with $n\ge 2$, and let $\A=\{0\}$. 
Then the complement $U=\C^n\mysetminus \{0\}$ is not Stein, and  
also not an abelian duality space, since $U\simeq S^{2n-1}$.   
Nevertheless, complements of hypersurfaces in Stein manifolds are 
again Stein \cite[Satz~1]{DG}.   
\end{remark}

\subsection{Vanishing of twisted cohomology}
\label{subsec:twist}

As an application of Theorem \ref{thm:mcm}, we can now prove 
the first theorem from the Introduction.

\begin{proof}[Proof of Theorem \ref{thm:intro1}]
Let $G=\pi_1(U)$.
From the definition, we need to show that $H^p(U,A)=0$ for $p\neq n$, and
$H^n(U,A)$ is torsion free, for $A=\Z[G]$ and $A=\Z[G^{\ab}]$.

For this, we need to check that $A$ is a $\MCM$ module.  By our injectivity
result (Lemma~\ref{lem:inj}), the restriction of $A$ to $\Z[C_{S,X}]$ is
a free module, for each local free abelian group $C_{S,X}$, and free
modules are $\MCM$.  Applying Theorem~\ref{thm:mcm} completes 
the proof.
\end{proof}

Another application is the following vanishing result, which proves 
Theorem \ref{thm:genvanish} from the Introduction. 
Recall that, for each stratum $X\in L(\A)$ one can choose a building set
$\G_X$: for each $g\in \G_X$ there is a corresponding boundary divisor
in the wonderful model for $T \A_X$, as well as a generator 
for each free abelian group $C_S$ for which $g\in S$.  Let $\gamma_g\in G$
denote the image of such a generator under the homomorphism $\alpha_{X,S}$
from Lemma~\ref{lem:inj}.

\begin{corollary}
\label{cor:generic vanishing}
Let $\A$ be an arrangement of hypersurfaces satisfying the hypotheses of
Theorem~\ref{thm:mcm}.  Suppose $A$ is a finite-dimensional representation
of $G=\pi_1(M(\A))$ over a field $\k$.  If $A^{\gamma_g}=0$ for all
$g\in \bigcup_{X\in L(\A)}\G_X$, then $H^i(M(\A),A)=0$ for all $0\leq i<n$.
\end{corollary}

\begin{proof}
First we note that the hypothesis makes sense: $\gamma_g\in G$
is only defined up to conjugacy; however, the property that $1$ is not an
eigenvalue of the representating matrix for $\gamma_g$ is invariant under
conjugation.  As noted in \cite[Lem.~5.7]{DSY1}, the hypothesis implies
that the $\k[G]$-module $A$ is $\MCM$, and Theorem~\ref{thm:mcm} applies.
\end{proof}

\begin{remark}
\label{rem:generic vanishing}
The vanishing of cohomology in Corollary~\ref{cor:generic vanishing} is 
generic, in the following sense.  For each $g$ in some building set $\G_X$, 
the representations $A\in \Hom(G,\GL_r(\k))$ for which $A^g=0$ form a 
Zariski open set. Since there are finitely many such $g$, the representations 
$Z\subseteq \Hom(G,\GL_r(\k))$ that satisfy the hypotheses of the 
Corollary above form an open subvariety.  If $Z$ is nonempty and the 
representation variety is irreducible, then, $Z$ is Zariski-dense.
\end{remark}

\subsection{Vanishing of $L^2$-cohomology}
\label{subsec:l2}
Finally, we show that the same hypotheses also imply
Theorem \ref{thm:l2vanish} from the Introduction. 
We refer to the survey of Eckmann~\cite{Eck} and
the book of L\"{u}ck~\cite{Lueck} for background on the subject.  
Related vanishing results can be found in \cite{DJLO11, DS13} 
(but see the caveat from the Introduction), as well as in the recent 
preprints \cite{LMWa, LMWb}.

Once again, let $\A$ be an arrangement of hypersurfaces in a 
compact, complex, smooth variety $Y$ of dimension $n$, and 
let $G=\pi_1(M(\A))$.  

\begin{theorem}
\label{thm:l2}
Suppose that $M(\A^X)$ is Stein for each $X\in L(\A)$.  Then the reduced 
$L^2$-cohomology groups $\tensor[^\rd]{H}{^i}(M(\A),\ell_2(G))$ vanish 
for all $i\neq n$.
\end{theorem}

\begin{proof}
We follow the approach of \cite{DJL07}, which we explain briefly here.
Let $\VN(G)$ denote the group von Neumann algebra of $G$; that is, the
algebra of (right) $G$-invariant bounded endomorphisms of $\ell_2(G)$.  
We have an algebra homomorphism $\C[G]\to \VN(G)$ sending $g\in G$ to 
left-multiplication by $g$.
Instead of working as usual in the abelian category of $\C[G]$-modules, we
let $\cT$ be the subcategory of $\N(G)$-modules of dimension zero, and 
$\cE$ be the (Serre) quotient of $\N(G)$-modules by $\cT$.  Then 
it is known that $\cE$ is an abelian category and the quotient construction
is exact, so our spectral sequence computation of $H^\hdot(M(\A),\ell_2(G))$
may be carried out in $\cE$.

To show that the reduced $L^2$-cohomology of the universal cover of $M(\A)$
is concentrated in cohomology degree $n$, then, we compute in $\cE$.  
For this, suppose 
$C_{S,X}$ is a free abelian subgroup generated by $\gamma_1,\ldots, \gamma_r$
as in Definition~\ref{def:mcm-arr}, where $r=\abs{S}$.  Let $R=\C[C_{S,X}]$.
In order to imitate the proof of Theorem~\ref{thm:mcm}, we will show that
\begin{equation}
\label{eq:L2 torus vanishing}
\tensor[^\rd]{H}{^i}(C_{S,X},\ell_2(G))=0 \quad\text{in $\cE$, for all $i$.}
\end{equation}

For this, let $V=(\gamma_1-1)\ell_2(G)$, and we show that
$\overline{V}=\ell_2(G)$, where
$\overline{\phantom{a}}$ denotes closure in the $\ell_2$-topology.
Noting that $\ell_2(G)$ is a Hilbert space with an orthonormal basis
$\set{g\colon g\in G}$, we can do this by checking that
$\overline{V}^\perp=V^\perp=0$, as follows. 
For any $c\in \ell_2(G)$, we may write 
\begin{equation}\label{eq:L2 c}
c=\sum_{i\in \Z, \, \Z g\in \Z\backslash G}c_{g,i}\gamma_1^ig\, ,
\end{equation}
for some coefficients $c_{g,i}\in \C$, choosing right coset representatives for
$\Z\cong\angl{\gamma_1}$.  If $c\in V^\perp$, then
\begin{eqnarray*}
0 & = & \angl{c,\gamma_1^i g-\gamma_1^{i-1} g} \quad\text{for all $i\in\Z$,
and $\Z g\in \Z\backslash G$,}\\
&=&c_{g,i}-c_{g,i-1}.
\end{eqnarray*}
So, for each $g$, the coefficient $c_{g,i}$ is independent of $i$.  Since by assumption
$\sum \abs{c_{g,i}}^2<\infty$, it must be the case that $c_{g,i}=0$ for all $i$,
which implies $c=0$.

Similarly, we show that, for $c\in \ell(G)$, if 
$(\gamma_1-1)c=0$, then $c=0$.  Writing $c$ as in \eqref{eq:L2 c}, we see 
then $\gamma_1c=c$ if and only if $c_{g,i}=c_{g,i+1}$ for all $i$.  Once
again, square-summability implies each $c_{g,i}=0$.

Combining, we obtain a short exact sequence
\begin{equation}\label{eq:l2ses}
\xymatrix{
0\ar[r] &\ell_2(G)\ar[r]^{\gamma_1-1} & \ell_2(G) \ar[r] & \ell_2(G)/V\ar[r] & 0,
}
\end{equation}
where $\ell_2(G)/V\in \cT$.  
Applying $H^\hdot(C_{S,X},-)$ to \eqref{eq:l2ses} gives 
a long exact sequence which reduces to isomorphisms
\[
\xymatrixcolsep{16pt}
\xymatrix{0\ar[r]&
\tensor[^\rd]{H}{^i}(C_{S,X},\ell_2(G))\ar[rr]^{H(\gamma_1-1)} && 
\tensor[^\rd]{H}{^i}(C_{S,X},\ell_2(G))\ar[r]&0
}
\]
in the quotient category $\cE$ for each $i$.

On the other hand, $\gamma_1-1$ is in the augmentation ideal of 
$C_{S,X}$, so it acts by zero on $H^i(C_{S,X},A)$ for any coefficients
module $A$, hence also on its image in $\cE$.  It follows 
that $\tensor[^\rd]{H}{^i}(C_{S,X},\ell_2(G))=0$, for each $i$, as required.
\end{proof}

\section{Applications}
\label{sect:apps}

\subsection{Linear, toric, and elliptic arrangements}
\label{subsec:toric}
We need the following consequence of a result due to De Concini and 
Gaiffi~\cite{DG16}.

\begin{proposition}
\label{prop:toricCpct}
If $\A$ is a toric arrangement in $\T:=(\C^*)^n$, there exists a 
compactification $Y$ and an arrangement of subvarieties $\L$ in $Y$
for which $U(\A)=Y\mysetminus \bigcup_{K\in\L}K$ and, for each connected
stratum $X\in L(\L)$, the relative interior $U(\A^X)\subseteq X$ is 
a Stein manifold.
\end{proposition}

\begin{proof}
De Concini and Gaiffi \cite{DG16} construct a compact toric variety $Y=Y_\Delta$
for which $U(\A)=Y\mysetminus \bigcup_{K\in\L}K$ and $\L$ is an arrangement of 
subvarieties.  Given a stratum $X\in L(\L)$, we may write it as
$X=\overline{gK}\cap Y_C$, where $gK\in L(\A)$ is a coset of a torus
in $\T$, and $Y_C$ is a closed torus
orbit in $Y$, indexed by a (closed) cone $C$ in $\Delta$, the rational fan of
$Y_\Delta$.  

Let $N=\Hom(\C^*,\T)$ denote the lattice of $1$-parameter subgroups, and
let $V_K$ denote the subspace of $N\otimes_{\Z}\R$ given by restriction to
the subtorus $K$.  By \cite[Thm.\ 3.1]{DG16}, the closure $\overline{gK}$ is 
again a toric variety, and $X$ is an affine toric subvariety corresponding
to the chamber $C$ in $V_K$ \cite[Prop.\ 3.1]{DG16}.
Then
\begin{gather}
\begin{aligned}
U(\A^X) &= X \mysetminus \Big(\bigcup_{
\substack{F\in\L\colon \\
F\not\supseteq gK}}
\overline{F} \cup
\bigcup_{\substack{C'\in \Delta\colon\\
 C'\not\subseteq C}} Y_{C'} \Big)\\
&= gK \mysetminus \bigcup_{
\substack{F\in\L\colon \\
F\not\supseteq gK}} F , 
\end{aligned}
\end{gather}
which is the complement of a toric arrangement in the torus $gK$.
Since toric arrangement complements are affine varieties, we conclude 
that $U(\A^X)$ is a Stein manifold.
\end{proof}

\begin{remark}
\label{rem:non-stein-bis}
The proof of Proposition~\ref{prop:toricCpct} relies on a special recursive
property of the construction from \cite{DG16}.  In general, though, 
complements of hypersurfaces in toric varieties need not be 
Stein (see, however, \cite{Ya} for an interesting case in which they are.)  
\end{remark}

\begin{remark}
\label{rem:stein}
It should also be noted that Stein 
manifolds need not be abelian duality spaces. Indeed, let 
$X$ be the complement of a hypersurface in $\C^n$ with $k$ components, where
$k<n$, $n\ge 2$, and $b_n(X)>0$.  Then $X$ is Stein and has homological 
dimension $n$, but $b_1(X)=k$, and so $X$ cannot be an abelian duality 
space of dimension $n$, by \cite[Prop.~5.9]{DSY2}.  

As a concrete example, let $X$ be the complement of the irreducible hypersurface 
$\C^3$ defined by the Brieskorn polynomial $f=x^3+y^3+z^3$. Then the 
Milnor fiber of $f$ is homotopic to a wedge of $8$ spheres, while the 
characteristic polynomial of the algebraic monodromy is $\Delta(t)=
(t-1)^2(t^2+t+1)^3$. The Wang exact sequence of the Milnor fibration 
now shows that $b_1(X)=1$ and $b_2(X)=b_3(X)=2$.
\end{remark}

We are now in a position to prove the second theorem from the Introduction.

\begin{proof}[Proof of Theorem \ref{thm:intro2}]
\eqref{a1}
If $M(\A)\subseteq\C^n$ is the complement of an affine-linear arrangement,
by adding a hyperplane, it is also the complement of an arrangement in 
$\PP^n$, and this result appeared in \cite[Thm.\ 5.6]{DSY1}.  We note that
central arrangement complements are a special case.

\eqref{a2}
This result was also reported in \cite[Cor.~6.4]{DSY1}; however, the 
proof given there is incomplete.  One reduces to the essential case and notes
that the restrictions $\A^X$ are again essential, hence Stein, by
\cite[Prop.\ 6.1]{DSY1}.  To verify that $\Z[G]$ and $\Z[G^{\ab}]$ 
are $\MCM$, one needs to know that the maps 
$\alpha_{X,S}\colon C_{X,S}\to G$
are injective, the proof of which was omitted, but 
is now covered by Lemma \ref{lem:inj}.

\eqref{a3}
By Proposition~\ref{prop:toricCpct}, the toric arrangement complement
admits a compactification which satisfies the conditions of 
Theorem~\ref{thm:intro1}.  The claim follows.
\end{proof}

\subsection{(Orbit) configuration spaces of Riemann surfaces}
\label{subsec:conf}

Let $\Gamma$ be a discrete group that acts freely and properly 
discontinuously on a space $X$.  The {\em orbit configuration 
space}\/ $F_{\Gamma}(X,n)$ is, by definition, the subspace of 
the cartesian product $X^{\times n}$ consisting of $n$-tuples
$(x_1,\ldots,x_n)$ for which the $\Gamma$-orbits of $x_i$ and 
$x_j$ are disjoint for all $1\leq i\neq j\leq n$.  If 
$\abs{\Gamma}=1$, then $F_{\Gamma}(X,n)=F(X,n)$, the
classical (ordered) configuration space.  Orbit configuration spaces were
first investigated in the thesis of Xicot\'{e}ncatl~\cite{Xi97}, and further 
studied, for instance, in \cite{CKX, Ca16}. 

In the case when $X=M$ is a smooth manifold of dimension $d$, and $\Gamma$ 
acts by diffeomorphisms, the orbit configuration space $F_{\Gamma}(M,n)$ is a
smooth manifold of dimension $dn$.  Perhaps the most studied case (and the 
one we are mainly interested here) is when $M=\Sigma_{g,k}$ 
is a Riemann surface of genus $g$ with $k\ge 0$ punctures, 
and $\Gamma$ is finite. 
Note that, if $k=0$, the complement in $\Sigma_g^{\times n}$ of 
$F_{\Gamma}(\Sigma_g,n)$ is the union 
of an arrangement of smooth, complex algebraic hypersurfaces. 

Xicot\'{e}ncatl showed that the classical Fadell--Neuwirth \cite{FN} fibration
applies in the more general case of orbit configuration spaces:
\begin{equation}
\label{eq:fn}
\xymatrixcolsep{20pt}
\xymatrix{
F_{\Gamma}(\Sigma_{g,k+\abs{\Gamma}},n-1) \ar[r]& F_{\Gamma}(\Sigma_{g,k},n)  
\ar[r]&\Sigma_{g,k}}.
\end{equation}

Consider the `tautological' compactification of the orbit configuration space 
$U=F_{\Gamma}(\Sigma_{g,k},n)$, namely $Y=\Sigma_g^{\times n}$.  The 
components of the boundary divisor, $D=Y\mysetminus U$, form an 
arrangement of hypersurfaces, 
\begin{equation}
\label{eq:bn}
\B_n := \set{H_{ij}^\gamma\mid \gamma\in \Gamma,1\leq i\neq j\leq n}\cup
\set{K_{i,l}\mid 1\leq i\leq n, 1\leq l\leq k},
\end{equation}
where $H_{ij}^\gamma$ is given by the equation $x_i=\gamma\cdot x_j$ and
$K_{i,l}$ by $x_i=p_l$, where $p_1,\ldots,p_k\in\Sigma_g$ are the punctures 
of $\Sigma_{g,k}$.

The intersection poset $L(\B_{n})$ can be described in terms of labelled
partitions via a slight generalization of the Dowling lattice.  To describe it, 
we write $\Pi\vDash[n]$ to denote a set partition $\Pi$ of $[n]$.
We regard $\Pi$ as a category whose objects are the set $[n]$, and 
for which there is a (unique) morphism $i\to j$ if and only if
$i\sim j$ in $\Pi$.
The action of $\Gamma$ on $\Sigma_{g,k}$ extends continuously to an 
action on $\Sigma_g$.
We let $\CC:=\CC(\Gamma,\Sigma_g)$ denote the category whose objects
are $\set{\Sigma_{g,k}, \set{p_1},\ldots,\set{p_k}}$, a $\Gamma$-equivariant
partition of $\Sigma_g$.  For a pair of objects $S$, $T$, by definition
$[\gamma]\colon S\to T$ is a morphism whenever $[\gamma]\in \Gamma/\stab(S)$
and $\gamma(S)\subseteq T$, and composition induced by the group operation.

Given a point $x\in Y=\Sigma_g^{\times n}$, let $\Pi_x$ 
be the partition determined by $x$ for which 
$i\sim j$ if and only if $(\Gamma\cdot x_i)\cap (\Gamma\cdot x_j)$ is nonempty.
Then $x_i=\gamma_{ij}x_j$ for some $\gamma_{ij}\in \Gamma$.
For a point $x\in \Sigma_g$, we let $[x]$ be the object of $\CC$ containing
$x$.  Define a functor $f_x\colon \Pi_x\to\CC$ by $f(i)=[x_i]$, 
for each $i\in [n]$,
and let $f_x(j\to i)=[\gamma_{ij}]$, for all $i\sim j$.  It is easy to check
that $f_x$ is well-defined, and has the following property.
\begin{lemma}
\label{lem:blocks}
For all $i,j\in [n]$, 
a point $x\in Y$ is on the hypersurface $H_{ij}^{\gamma}$
if and only if $f_x(j\to i)=[\gamma_{ij}]$.  A point $x\in Y$ is on the
hypersurface $K_{i,l}$ if and only if $f_x(i)=\set{p_l}$.
\end{lemma}

Now we can describe the stratification of $Y$.  Let
\begin{equation}
\label{eq:triples}
\pairs_n = \set{(\Pi,f)\colon \Pi\vDash[n],\,
f\in\Funct(\Pi,\CC)}.
\end{equation}
Consider the function $\Phi\colon Y\to \pairs_n$ given by $\Phi(x)=(\Pi_x,f_x)$.
Let $X_{\Pi,f}=\Phi^{-1}(\Pi,f)$ for each pair $(\Pi,f)\in \pairs_n$.

\begin{proposition}
\label{prop:strata}
For any pair $(\Pi,f)\in\im(\Phi)$, 
the space $X_{\Pi,f}$ coincides with $M(\B_n^X)$, where
\begin{align}
X & =\overline{X_{\Pi,f}} \nonumber \\
 & = \bigcap_{\substack{i,j,\gamma\colon\\
f(j\to i)=[\gamma_{ij}]}}H_{ij}^{\gamma}\cap
\bigcap_{\substack{i,l\colon\\
f(i)=\set{p_l}}}K_{i,l} \label{eq:X pi f}
\end{align}
\end{proposition}

\begin{proof}
By Lemma~\ref{lem:blocks}, a point $x$ belongs to $X_{\Pi,f}$ if and only if it is 
in the intersection \eqref{eq:X pi f}, and $x$ is not contained in any other
hypersurface.  But this is exactly the space $M(\B_n^X)$.
\end{proof}

\begin{proposition}
\label{prop:Sigma}
If $k>0$, then for each $X\in L(\B_n)$, the complement $M(\B_{n}^X)$ is 
a Stein manifold.
\end{proposition}

\begin{proof}
First we check the case $X=Y$, where $\B_{n}^X=\B_{n}$.  Here, 
$M(\B_n)$ is a hypersurface complement in $\Sigma_{g,k}^{\times n}$.
As mentioned previously, the product of open Riemann surfaces is Stein, 
and a hypersurface complement in a Stein manifold is again Stein; 
thus, the base case is proved.

In general, for a stratum $X_{\Pi,f}$ defined above, 
it is easy to see from Proposition~\ref{prop:strata} that the arrangement
$\B_n^X$ is again an orbit configuration space, $F_\Gamma(\Sigma_{g,k},m)$,
where $m$ is the number of diagonal blocks of $\Pi$.
This reduces the claim to the case $X=Y$ treated above.
\end{proof}

Now we may characterize the duality properties of orbit configuration spaces 
of points in Riemann surfaces.
\begin{proof}[Proof of Theorem~\ref{thm:intro3}]
For $g>0$ and $k>0$, we see that $F_{\Gamma}(\Sigma_{g,k},n)$ is both a duality
space and an abelian duality space of dimension $n$: by Proposition~\ref{prop:Sigma}, 
the configuration space satisfies the hypotheses of Theorem~\ref{thm:intro1}.

For $k=0$, we note that $\Sigma_{g}$ is a (Poincar\'e) 
duality space of dimension $2$.
Noting also that $F_{\Gamma}(\Sigma_{g,1},1)=\Sigma_{g,1}$, we see that 
$F_{\Gamma}(\Sigma_{g},n)$ is a duality space of dimension $n+1$ for all 
$n\geq1$ by induction, using
the fibration sequence \eqref{eq:fn} and a classical result of Bieri
and Eckmann~\cite[Theorem~3.5]{BE}.

If $k=0$ and $g=1$, the configuration space 
$F_{\Gamma}(\Sigma_{g},n)$ is an elliptic arrangement complement,
hence an abelian duality space by Theorem~\ref{thm:intro2}.
However, if $k=0$ and $g\geq2$, it need not be an abelian duality space.
The (easy) case $n=1$ appears as \cite[Example\ 5.8]{DSY2}: $\Sigma_{g,k}$ is
an abelian duality space if and only if $k\geq 1$.
For $n\ge 2$, we restrict our attention to the case 
where $\Gamma$ is trivial (and $k=0$).

Again from the fibration sequence \eqref{eq:fn} and by induction 
on $n$ we see that $F(\Sigma_g,n)$ is a CW-complex of dimension 
at most $n+1$. 
Furthermore, it is known that $b_3(F(\Sigma_g,2))=2g$ (see e.g.~\cite[Cor.~12]{Azam}) 
and that $b_{n+1}(F(\Sigma_g,n))\ne 0$, for all $n\ge 3$ (see \cite[Prop.~1.4]{Maguire}). 
Thus, if $F(\Sigma_g,n)$ were to be an abelian duality space, it would 
have to be of (formal) dimension $n+1$. 
On the other hand, by \cite{FT}, the generating series for the Euler 
characteristics of the unordered configuration spaces of $\Sigma_g$ 
is given by 
\begin{equation}
\label{eq:ft}
1+ \sum_{n\ge 1} \chi(C(\Sigma_g,n)) t^n = (1+t)^{2-2g}.
\end{equation}
Therefore, $(-1)^n  \chi(C(\Sigma_g,n)) >  0$, and hence 
$(-1)^n  \chi(F(\Sigma_g,n)) > 0$, also.  In view of \eqref{eq:signedeuler},
we conclude that $F(\Sigma_g,n)$ is not an abelian 
duality space.

Last, we consider the genus $g=0$ case.  If $k>0$, then 
$F_{\Gamma}(\Sigma_{0,k},n)$
is an affine-linear hyperplane arrangement complement of corank $0$.
By Theorem~\ref{thm:intro2} again, it is both a duality space and an abelian
duality space of dimension $n$. 

It remains to show that if $g=0$ and $k=0$, then $F(\Sigma_0,n)$ is
neither a duality nor an abelian duality space.  For $n\leq 2$, this is clear,
since then $F(\Sigma_0,n)\simeq S^2$ from \eqref{eq:fn}.  
For $n\ge 3$, we have (see, e.g., \cite{FZ})
\begin{equation}
\label{eq:fz}
F(\Sigma_0,n) \cong \PGL_2(\C)\times F(\Sigma_{0,2},n-2),
\end{equation}
where $F(\Sigma_{0,2},n-2)$ is the complement of an affine-linear hyperplane
arrangement in $\C^{n-2}\subseteq \PP^{n-2}$.  
This is both a duality space and an 
abelian duality space but $\PGL_2(\C)\simeq \SO(3)$ 
is neither, from which the conclusion follows.
\end{proof}

\begin{ack}
We would like to thank Giovanni Gaiffi for helping us apply his work
with Corrado De Concini to our Proposition~\ref{prop:toricCpct}, 
and Stefan Papadima for help with Lemma~\ref{lem:surj}.  We also 
thank the referees for their careful reading of our manuscript and 
for their many insightful comments and suggestions. 
We gratefully acknowledge support by the Research in Pairs program 
during our stay at the Mathematisches Forschungsinstitut Oberwolfach 
in May--June 2016; the
first author would also like to thank the University of Sydney School of
Mathematics and Statistics for its hospitality.
\end{ack}

\newcommand{\arxiv}[1]
{\texttt{\href{http://arxiv.org/abs/#1}{arXiv:#1}}}
\newcommand{\arx}[1]
{\texttt{\href{http://arxiv.org/abs/#1}{arXiv:} 
\href{http://arxiv.org/abs/#1}{#1}}}
\newcommand{\doi}[1]
{\texttt{\href{http://dx.doi.org/#1}{\nolinkurl{doi:#1}}}}

\renewcommand{\MR}[1]
{\href{http://www.ams.org/mathscinet-getitem?mr=#1}{\textbf{MR}\:#1}}
\newcommand{\MRh}[2]
{\href{http://www.ams.org/mathscinet-getitem?mr=#1}{\textbf{MR}\:#1 (#2)}}

\bibliographystyle{amsalpha}
\providecommand{\bysame}{\leavevmode\hbox to3em{\hrulefill}\thinspace}
\providecommand{\MR}{\relax\ifhmode\unskip\space\fi MR }
\providecommand{\MRhref}[2]{%
  \href{http://www.ams.org/mathscinet-getitem?mr=#1}{#2}
}
\providecommand{\href}[2]{#2}

\end{document}